\documentclass[a4paper,11pt,reqno]{amsart}

\usepackage[utf8]{inputenc}
\usepackage[english]{babel}
\usepackage{amsfonts}
\usepackage{amsmath}
\usepackage{amssymb}
\usepackage{amsthm}
\usepackage{enumerate}
\usepackage{geometry}

\newcommand*{\1}{1\!\!\,\mathrm{I}}
\newcommand*{\abs}[1]{\left|#1\right|}
\newcommand*{\norm}[1]{\left\|#1\right\|}
\newcommand*{\Prob}[1]{\mathbf{P} \left\lbrace #1\right\rbrace}

\newcommand*{\as}{\text{a.~s.}}
\newcommand*{\qv}[1]{\left\langle #1\right\rangle}
\newcommand*{\jqv}[2]{\left\langle #1,#2\right\rangle}

\newcommand*{\ve}{\varepsilon}
\newcommand*{\vf}{\varphi}
\newcommand*{\mbN}{\mathbb{N}}
\newcommand*{\mbQ}{\mathbb{Q}}
\newcommand*{\mbR}{\mathbb{R}}
\newcommand*{\mcB}{\mathcal{B}}
\newcommand*{\mcF}{\mathcal{F}}

\theoremstyle{plain}
\newtheorem{theorem}{Theorem}
\newtheorem{lemma}[theorem]{Lemma}
\newtheorem{proposition}[theorem]{Proposition}

\theoremstyle{definition}
\newtheorem{definition}[theorem]{Definition}

\theoremstyle{remark}

\title
[Weak convergence of Harris flows]
{A note on weak convergence of\\ the $n$-point motions of Harris flows}
\author{V.~V.~Fomichov}
\address{Vladimir~Fomichov: Institute of Mathematics, National Academy of Sciences of Ukraine, Tereshchenkivska str.~3, Kiev~01004, Ukraine}
\email{v-vfom@imath.kiev.ua}

\begin{document}

\begin{abstract}
In this note we extend the main results of~\cite{Dorogovtsev2004} and~\cite{Malovichko}, which concern the weak convergence of the $n$-point motions of smooth Harris flows to those of the Arratia flow, to the case when the covariance functions of these Harris flows converge pointwise to a covariance function whose support is of zero Lebesgue measure.
\end{abstract}

\maketitle

The main aim of this note is to generalize the results of~\cite{Dorogovtsev2004} and~\cite{Malovichko} concerning the weak convergence of the $n$-point motions of Harris flows.

We begin by recalling the definition of a Harris flow (e.~g., see~\cite[Definition~1.2]{DorogovtsevFomichov}).

\begin{definition}
A random field $\{x(u,t),\; u\in\mbR,\; t\geqslant 0\}$ is called a \emph{Harris flow} with covariance function $\Gamma$ if it satisfies the following conditions:
\begin{enumerate}
\item[(i)]
for any $u\in\mbR$ the stochastic process $\{x(u,t),\; t\geqslant 0\}$ is a Brownian motion with respect to the common filtration $(\mcF_t:=\sigma\{x(v,s),\; v\in\mbR,\; 0\leqslant s\leqslant t\})_{t\geqslant 0}$ such that $x(u,0)=u$;

\item[(ii)]
for any $u,v\in\mbR$, if $u\leqslant v$, then $x(u,t)\leqslant x(v,t)$ for all $t\geqslant 0$;

\item[(iii)]
for any $u,v\in\mbR$ the joint quadratic variation of the martingales $\{x(u,t),\; t\geqslant 0\}$ and $\{x(v,t),\; t\geqslant 0\}$ is given by
$$
\jqv{x(u,\cdot)}{x(v,\cdot)}_t=\int\limits_0^t \Gamma(x(u,s)-x(v,s))\,ds,\quad t\geqslant 0.
$$
\end{enumerate}
\end{definition}

Note that the function $\Gamma$ is necessarily non-negative definite and, in particular, symmetric. Besides, without loss of generality we always assume that
$$
\Gamma(0)=1,
$$
so that the one-point motions of Harris flows we consider are standard Brownian motions.

The existence of random fields satisfying the above conditions~(i), (ii) and (iii) under mild assumptions on the covariance function was proved in~\cite{Harris}.

A Harris flow with covariance function $\Gamma=\1_{\{0\}}$ is called the \emph{Arratia flow} (here $\1_A(z)\equiv\1\{z\in A\}$ stands for the indicator function of the set $A$). It is one of the first examples of Harris flows and was initially constructed in~\cite{Arratia} as the weak limit of a family of coalescing simple random walks. Throughout this paper the Arratia flow will be denoted by $\{x_0(u,t),\; u\in\mbR,\; t\geqslant 0\}$.

It is convenient to construct Harris flows with a smooth covariance function as solutions of stochastic differential or integral equations. To be more precise, let us consider the following stochastic integral equation:
\begin{equation}
\label{equation1}
x(u,t)=u+\int\limits_0^t \int\limits_\mbR \vf(x(u,s)-q)\,W(dq,ds),\quad t\geqslant 0,
\end{equation}
where $u\in\mbR$ plays the role of a parameter, $W$ is a Wiener sheet on $\mbR \times [0;+\infty)$ and the function $\vf\in C_0^\infty(\mbR,[0;+\infty))$ (i.~e. infinitely differentiable and with compact support) is symmetric and has a unit $L_2$-norm.

It is known~\cite{Dorogovtsev2004} that under these conditions on the function $\vf$ this equation has a unique strong solution for every $u\in\mbR$ and the random field $\{x(u,t),\; u\in\mbR,\; t\geqslant 0\}$ is a Harris flow with covariance function $\Gamma$ given by
$$
\Gamma(z):=\int\limits_\mbR \vf(z+q)\vf(q)\,dq,\quad z\in\mbR.
$$

Now we can formulate the main results of~\cite{Dorogovtsev2004} and~\cite{Malovichko}. Although these results were proved for the case of the finite time interval $[0;1]$, their proofs remain valid for the more general case of the infinite time interval $[0;+\infty)$ and it is in this form that we formulate them below.

\begin{theorem}
\cite[Theorem~3]{Dorogovtsev2004}
For $\ve>0$ define
\begin{equation}
\label{equation2}
\vf_\ve(q):=\dfrac{1}{\sqrt{\ve}} \vf\left(\dfrac{q}{\ve}\right),\quad q\in\mbR,
\end{equation}
and let $\{x_\ve(u,t),\; u\in\mbR,\; t\geqslant 0\}$ be the Harris flow formed by the solutions of the stochastic integral equation~\eqref{equation1} with $\vf_\ve$ instead of $\vf$. Then for any $n\in\mbN$ and for any $u_1,\ldots,u_n\in\mbR$ the weak convergence
\begin{gather*}
(x_\ve(u_1,\cdot),\ldots,x_\ve(u_n,\cdot)) \stackrel{w}{\longrightarrow} (x_0(u_1,\cdot),\ldots,x_0(u_n,\cdot)),\quad \ve\to 0+,
\end{gather*}
takes place in the space $C([0;+\infty),\mbR^n)$.
\end{theorem}

Note that in this case for the covariance function $\Gamma_\ve$ of the Harris flow $x_\ve$ we have
$$
\forall\, z\in\mbR:\quad \Gamma_\ve(z) \longrightarrow \1_{\{0\}}(z),\quad \ve\to 0+,
$$
and also
\begin{equation}
\label{equation3}
\vf_\ve^2 \longrightarrow \delta_0,\quad \ve\to 0+,
\end{equation}
in the sense of generalized functions (here and below $\delta_a$ stands for the delta function at point $a\in\mbR$).

In~\cite{Malovichko} it was shown that the assertion of this theorem still holds true even if $\vf_\ve^2$ converges to a generalized function distinct from $\delta_0$.

\begin{theorem}
\cite[p.~1538]{Malovichko}
For $\ve>0$ define
$$
\vf_\ve(q):=\dfrac{\sqrt{\alpha}}{\sqrt{\ve}} \vf\left(\dfrac{q-a_1}{\ve}\right)+\dfrac{\sqrt{\beta}}{\sqrt{\ve}} \vf\left(\dfrac{q-a_2}{\ve}\right),\quad q\in\mbR,
$$
where $0<\alpha,\beta<1$, $\alpha+\beta=1$, and $a_1<a_2$, and let $\{x_\ve(u,t),\; u\in\mbR,\; t\geqslant 0\}$ be the Harris flow formed by the solutions of the stochastic integral equation~\eqref{equation1} with $\vf_\ve$ instead of $\vf$. Then for any $n\in\mbN$ and for any $u_1,\ldots,u_n\in\mbR$ the weak convergence
$$
(x_\ve(u_1,\cdot),\ldots,x_\ve(u_n,\cdot)) \stackrel{w}{\longrightarrow} (x_0(u_1,\cdot),\ldots,x_0(u_n,\cdot)),\quad \ve\to 0+,
$$
takes place in the space $C([0;+\infty),\mbR^n)$.
\end{theorem}

Note that in this case for the covariance function $\Gamma_\ve$ of the Harris flow $x_\ve$ we have
$$
\forall\, z\in\mbR:\quad \Gamma_\ve(z) \longrightarrow \sqrt{\alpha\beta} \cdot \1_{\{-b\}}(z)+\1_{\{0\}}(z)+ \sqrt{\alpha\beta} \cdot \1_{\{+b\}}(z),\quad \ve\to 0+,
$$
where $b:=a_2-a_1$, and also
\begin{equation}
\label{equation4}
\vf_\ve^2 \longrightarrow \alpha\delta_{a_1}+\beta\delta_{a_2},\quad \ve\to 0+,
\end{equation}
in the sense of generalized functions.

Here we show that the proof presented in~\cite{Malovichko} can be extended to the case when the right-hand side of~\eqref{equation4} is replaced by a discrete probability measure on the real line satisfying some mild conditions. To be more precise, let $\nu$ be an arbitrary finite singular measure on the real line having at least one atom, i.~e. such that
$$
\nu^2(\Delta)>0,
$$
where
\begin{gather*}
\nu^2:=\nu \otimes \nu
\intertext{and}
\Delta:=\{\vec{q}=(q_1,q_2)\in\mbR^2 \mid q_1=q_2\}.
\end{gather*}
Suppose that the function $\vf$ considered above is additionally non-decreasing on $(-\infty;0]$ and is non-increasing on $[0;+\infty)$ and that $\vf_\ve$ is defined by~\eqref{equation2}. Let us set
$$
\psi_\ve(z):=c_\ve \int\limits_\mbR \varphi_\ve(z-q)\,\nu(dq),\quad z\in\mbR,
$$
where the constant $c_\ve>0$ is chosen to be such that
$$
\int\limits_\mbR \psi_\ve^2(z)\,dz=1.
$$
It is clear that
\begin{gather*}
c_\ve=\left[\iint\limits_{\mbR^2} \Phi_\ve(q_1-q_2)\,\nu^2(dq_1dq_2)\right]^{-1/2},
\intertext{where}
\Phi_\ve(z):=\int\limits_\mbR \vf_\ve(z+q)\vf_\ve(q)\,dq,\quad z\in\mbR,
\end{gather*}
and also
$$
\psi_\ve\in C^\infty(\mbR).
$$
For $\ve>0$ let $\{x_\ve(u,t),\; u\in\mbR,\; t\geqslant 0\}$ be the Harris flow formed by the solutions of the stochastic integral equation~\eqref{equation1} with $\psi_\ve$ instead of $\vf$. The covariance functions of these Harris flows are given by
$$
\Gamma_\ve(z):=\int\limits_\mbR \psi_\ve(z+q)\psi_\ve(q)\,dq,\quad z\in\mbR.
$$

The main result of this note is the following theorem.

\begin{theorem}
\label{theorem4}
For any $n\in\mbN$ and for any $u_1,\ldots,u_n\in\mbR$ the weak convergence
$$
(x_\ve(u_1,\cdot),\ldots,x_\ve(u_n,\cdot))\stackrel{w}{\longrightarrow} (x_0(u_1,\cdot),\ldots,x_0(u_n,\cdot)),\quad \ve\to 0+,
$$
takes place in the space $C([0;+\infty),\mbR^n)$.
\end{theorem}

Following~\cite{Malovichko} we divide the proof into several lemmas. We repeat the considerations of~\cite{Malovichko}, where necessary, as concisely as possible and omit the proofs which are similar to those of that paper. The main difference lies in the proof of Lemma~\ref{lemma10}, since the idea used in the proof of its analogue~\cite[Lemma~6]{Malovichko} cannot be applied to our case. Our proof of Lemma~\ref{lemma10} is based on additional Lemmas~\ref{lemma6} and~\ref{lemma9}.

Before proceeding to the proof of the main result, however, we prove an analogue of relations~\eqref{equation3} and~\eqref{equation4}. To formulate it, let $\nu_0$ be a discrete probability measure on the real line defined by
$$
\nu_0(A):=\dfrac{\sum\limits_{k\colon a_k\in A} (\nu(\{a_k\}))^2} {\sum\limits_{k} (\nu(\{a_k\}))^2},\quad A\in\mcB(\mbR),
$$
where $\{a_k,\; k\geqslant 1\}$ are the atoms of the measure $\nu$ and $\mcB(\mbR)$ is the Borel $\sigma$-field of the real line.

\begin{proposition}
For every function $f\in C_0^\infty(\mbR)$ we have
$$
\lim_{\ve\to 0+} \int\limits_\mbR f(z)\psi_\ve^2(z)\,dz=\int\limits_\mbR f(z)\nu_0(dz).
$$
\end{proposition}

\begin{proof}
By Fubini's theorem we have
\begin{gather*}
\lim_{\ve\to 0+} \int\limits_\mbR f(z)\psi_\ve^2(z)\,dz=\\
=\lim_{\ve\to 0+} \left[c_\ve^2 \iint\limits_{\mbR^2} \left(\int\limits_\mbR f(z) \varphi_\ve(z-q_1) \varphi_\ve(z-q_2)\,dz\right)\nu^2(dq_1dq_2)\right].
\end{gather*}
However, by the dominated convergence theorem
$$
\lim_{\ve\to 0+} c_\ve^2=\left[\iint\limits_{\mbR^2} \left(\lim_{\ve\to 0+} \Phi_\ve(q_1-q_2)\right) \nu^2(dq_1dq_2)\right]^{-1}= \dfrac{1}{\nu^2(\Delta)}.
$$
Moreover, since for any $q_1,q_2\in\mbR$ we have
$$
\abs{\int\limits_\mbR f(z) \varphi_\ve(z-q_1) \varphi_\ve(z-q_2)\,dz}\leqslant \norm{f}_\infty \cdot \Phi_\ve(q_1-q_2)\leqslant \norm{f}_\infty<+\infty,
$$
where
$$
\norm{f}_\infty:=\max_{z\in\mbR} \abs{f(z)},
$$
by the same theorem
$$
\lim_{\ve\to 0+} \int\limits_\mbR f(z) \varphi_\ve(z-q_1) \varphi_\ve(z-q_2)\,dz= f(q_1)\1\{q_1=q_2\}.
$$
It remains to note that
\[
\dfrac{1}{\nu^2(\Delta)} \iint\limits_{\mbR^2} f(q_1) \1\{q_1=q_2\}\,\nu^2(dq_1dq_2)=\int\limits_\mbR f(z)\,\nu_0(dz).
\qedhere
\]
\end{proof}

Now let us set
$$
\Delta_z:=\{\vec{q}=(q_1,q_2)\in\mbR^2 \mid q_1-q_2=z\},\quad z\in\mbR.
$$
Then it is easy to see that for every $z\in\mbR$ we have
$$
\lim_{\ve\to 0+} \Gamma_\ve(z)=\Gamma_0(z),
$$
where the function $\Gamma_0$ is given by
$$
\Gamma_0(z):=\dfrac{\nu^2(\Delta_z)}{\nu^2(\Delta_0)}.
$$

Moreover, the set
\begin{equation}
\label{equation5}
D:=\{z\in\mbR \mid \Gamma_0(z)>0\}
\end{equation}
is countable, since the family $\{\Delta_z,\; z\in\mbR\}$ is a partition of $\mbR^2$ and $\nu^2(\mbR^2)<+\infty$.

\begin{lemma}
\label{lemma6}
The following assertions hold true:
\begin{gather}
\label{equation6}
\lim\limits_{\abs{z}\to +\infty} \Gamma_0(z)=0,\\
\label{equation7}
\forall\, \delta>0:\quad \sup\limits_{\abs{z}\geqslant\delta} \Gamma_0(z)<1.
\end{gather}
\end{lemma}

\begin{proof}
To prove \eqref{equation6} note that for any $\ve>0$
$$
0\leqslant\Gamma_0(z)\leqslant\dfrac{1}{\nu^2(\Delta_0)} \iint\limits_{\mbR^2} \Phi_\ve(z+q_1-q_2)\,\nu^2(dq_1dq_2)
$$
and that by the dominated convergence theorem the last expression converges to zero as $\abs{z}\to +\infty$.

Now suppose that~\eqref{equation7} is false, i.~e. that there exists some $\delta_0>0$ such that
\begin{equation}
\label{equation8}
\sup\limits_{\abs{z}\geqslant\delta_0} \Gamma_0(z)=1.
\end{equation}
It means, in particular, that we can find some $z_1\geqslant\delta_0$ such that
$$
\Gamma_0(z_1)>\dfrac{1}{2}.
$$

Since the function $\Gamma_0$ is non-negative definite and $\Gamma_0(0)=1$, we have (e.~g., see~\cite[p.~22]{Kuo})
$$
\forall\, x,y\in\mbR:\quad \abs{\Gamma_0(x)-\Gamma_0(y)}\leqslant 2\sqrt{1-\Gamma_0(x-y)},
$$
and, in particular, for any $z\in\mbR$
\begin{equation}
\label{equation9}
\abs{\Gamma_0(z_1+z)-\Gamma_0(z_1)}\leqslant 2\sqrt{1-\Gamma_0(z)}.
\end{equation}
Using~\eqref{equation8} and~\eqref{equation9} and the symmetry of $\Gamma_0$ we can choose $z_2\geqslant \delta_0$ such that
\begin{gather*}
\abs{\Gamma_0(z_1+z_2)-\Gamma_0(z_1)}<\Gamma_0(z_1)-\dfrac{1}{2}
\intertext{and so}
\Gamma_0(z_1+z_2)>\dfrac{1}{2}.
\end{gather*}

Proceeding further in this way we obtain a sequence $\{z_n\}_{n=1}^\infty$ such that
\begin{gather*}
z_n\geqslant\delta_0,\quad n\geqslant 1,\\
\Gamma_0(z_1+\ldots+z_n)>\dfrac{1}{2},
\end{gather*}
which contradicts~\eqref{equation6}.
\end{proof}

Now fix arbitrary $n\in\mbN$ and $u_1,\ldots,u_n\in\mbR$, $u_1<\ldots<u_n$, and consider the family
$$
\{\vec{x}_\ve=(x_\ve(u_1,\cdot),\ldots,x_\ve(u_n,\cdot)),\; \ve>0\}
$$
of random elements in the space $C([0;+\infty),\mbR^n)$ endowed with the distance
\begin{gather*}
\rho(\vec{f},\vec{g}):=\sum_{i=1}^n \sum_{k=1}^\infty \dfrac{1}{2^k} \dfrac{\max\limits_{0\leqslant t\leqslant k} \abs{f_i(t)-g_i(t)}} {1+\max\limits_{0\leqslant t\leqslant k} \abs{f_i(t)-g_i(t)}},\\
\vec{f}=(f_1,\ldots,f_n)\in C([0;+\infty),\mbR^n),\\
\vec{g}=(g_1,\ldots,g_n)\in C([0;+\infty),\mbR^n).
\end{gather*}
Since all stochastic processes $\{x_\ve(u_i,t),\; t\geqslant 0\}$, $1\leqslant i\leqslant n$, are Wiener processes, thus having the same distribution in the complete separable metric space $C([0;+\infty),\mbR)$, using Prohorov's theorem one can easily show that the family $\{\vec{x}_\ve,\; \ve>0\}$ is weakly relatively compact. Let $\vec{x}=(x(u_1,\cdot),\ldots,x(u_n,\cdot))$ be one of its limit points (as $\ve\to 0+$).

\begin{lemma}
The $n$-dimensional stochastic process $\{\vec{x}(t),\; t\geqslant 0\}$ is a martingale (with respect to its own filtration).
\end{lemma}

\begin{proof}
The proof of this lemma is identical to that of~\cite[Lemma~2]{Malovichko} and is therefore omitted.
\end{proof}

\begin{lemma}
\label{lemma8}
With probability one for any $i,j\in \{1,\ldots,n\}$ we have
\begin{gather*}
0\leqslant\jqv{x(u_i,\cdot)}{x(u_j,\cdot)}_t-\jqv{x(u_i,\cdot)}{x(u_j,\cdot)}_s\leqslant\int\limits_s^t \Gamma_0(x(u_i,r)-x(u_j,r))\,dr,\\
0\leqslant s\leqslant t<+\infty.
\end{gather*}
\end{lemma}

\begin{proof}
Fix arbitrary $i,j\in \{1,\ldots,n\}$, $i\neq j$, and in the space $C([0;+\infty),\mbR^{n+1})$ consider the random elements
$$
\vec{x}_\ve^{(ij)}=(x_\ve(u_1,\cdot),\ldots,x_\ve(u_n,\cdot), \theta_\ve^{(ij)}),\quad \ve>0,
$$
where
$$
\theta_\ve^{(ij)}(t):=\jqv{x_\ve(u_i,\cdot)}{x_\ve(u_j,\cdot)}_t,\quad t\geqslant 0.
$$
As in the proof of~\cite[Lemma~3]{Malovichko} one can show that the family $\{\vec{x}_\ve^{(ij)},\; \ve>0\}$ is weakly relatively compact and that, if
$$
\vec{x}_{\ve_n}^{(ij)}\stackrel{w}{\longrightarrow}\vec{x}^{(ij)},\quad n\to\infty,
$$
in the space $C([0;+\infty),\mbR^{n+1})$ for some sequence $\{\ve_n\}_{n=1}^\infty$ strictly decreasing to zero, with $\vec{x}^{(ij)}:=(x(u_1,\cdot),\ldots,x(u_n,\cdot),\theta^{(ij)})$, then
$$
\theta^{(ij)}(t)=\jqv{x(u_i,\cdot)}{x(u_j,\cdot)}_t,\quad t\geqslant 0.
$$

Now, since the set
$$
\{\vec{f}=(f_1,\ldots,f_{n+1})\in C([0;+\infty),\mbR^{n+1}) \mid 0\leqslant f_{n+1}(t)-f_{n+1}(s)\leqslant \int\limits_s^t h_\delta(f_i(r)-f_j(r))\,dr\},
$$
where $0\leqslant s\leqslant t<+\infty$ and
$$
h_\delta(z):=\dfrac{1}{\nu^2(\Delta_0)} \iint\limits_{\mbR^2} \Phi_\delta(z+q_1-q_2)\,\nu^2(dq_1dq_2),\quad z\in\mbR,
$$
with $\delta>0$, is closed and
$$
\Gamma_\ve(z)=c_\ve^2 \iint\limits_{\mbR^2} \Phi_\ve(z+q_1-q_2)\,\nu^2(dq_1dq_2)\leqslant h_\delta(z),\quad z\in\mbR,
$$
for $\ve<\delta$, we obtain that
\begin{gather*}
\Prob{0\leqslant\theta^{(ij)}(t)-\theta^{(ij)}(s)\leqslant \int\limits_s^t h_\delta(x(u_i,r)-x(u_j,r))\,dr}\geqslant\\
\geqslant\varlimsup_{n\to\infty} \Prob{0\leqslant\theta_{\ve_n}^{(ij)}(t)- \theta_{\ve_n}^{(ij)}(s)\leqslant \int\limits_s^t h_\delta(x_{\ve_n}(u_i,r)-x_{\ve_n}(u_j,r))\,dr}=1.
\end{gather*}
Thus, for every $\delta>0$ with probability one
$$
0\leqslant\theta^{(ij)}(t)-\theta^{(ij)}(s)\leqslant \int\limits_s^t h_\delta(x(u_i,r)-x(u_j,r))\,dr,\quad 0\leqslant s\leqslant t<+\infty,\quad s,t\in\mbQ,
$$
and so with probability one
$$
0\leqslant\theta^{(ij)}(t)-\theta^{(ij)}(s)\leqslant \int\limits_s^t \Gamma_0(x(u_i,r)-x(u_j,r))\,dr,\quad 0\leqslant s\leqslant t<+\infty.
$$
The lemma is proved.
\end{proof}

\begin{lemma}
\label{lemma9}
With probability one for any $i,j\in \{1,\ldots,n\}$ we have
$$
\lim_{t\to +\infty} (x(u_i,t)-x(u_j,t))=0.
$$
\end{lemma}

\begin{proof}
Let us fix arbitrary $i,j\in \{1,\ldots,n\}$, $i>j$. The proof of the existence of the limit
$$
\lim_{t\to +\infty} (x(u_i,t)-x(u_j,t))
$$
is similar to the proof of~\cite[Lemma~1]{Lagunova}. Namely, we note (e.~g., see~\cite[Theorem~18.4]{Kallenberg}) that with probability one the following representation takes place:
\begin{equation}
\label{equation10}
x(u_i,t)-x(u_j,t)=(u_i-u_j)+\beta(\tau(t)),\quad t\geqslant 0,
\end{equation}
where $\{\beta(t),\; t\geqslant 0\}$ is a standard Wiener process (maybe defined on an extended probability space) and
\begin{equation}
\label{equation11}
\tau(t):=\qv{x(u_i,\cdot)-x(u_j,\cdot)}_t=
2t-2\jqv{x(u_i,\cdot)}{x(u_j,\cdot)}_t,\quad t\geqslant 0.
\end{equation}
Then
$$
x(u_i,t)-x(u_j,t)\geqslant 0,\quad t\geqslant 0,
$$
implies that
$$
\tau(t)\leqslant\overline{\tau},\quad t\geqslant 0,
$$
where
$$
\overline{\tau}:=\inf\{t\geqslant 0 \mid \beta(t)=-(u_i-u_j)\}<+\infty\quad \as
$$
Therefore, there exists the limit
$$
\lim_{t\to +\infty} \tau(t)=:\tau(+\infty)\leqslant\overline{\tau}
$$
and so, due to the continuity of $\beta$,
$$
\lim_{t\to +\infty} (x(u_i,t)-x(u_j,t))=(u_i-u_j)+\beta(\tau(+\infty)).
$$

Now suppose that
\begin{gather*}
\tau(+\infty)<\overline{\tau},
\intertext{i.~e.}
\lim_{t\to +\infty} (x(u_i,t)-x(u_j,t))>0.
\end{gather*}
Then there exists $\delta_0>0$ (depending on $\omega$) such that
$$
x(u_i,t)-x(u_j,t)>\delta_0,\quad t\geqslant 0.
$$
So using Lemma~\ref{lemma8} (with $s=0$) and Lemma~\ref{lemma6} we obtain that
\begin{align*}
\tau(t)&=2t-2\jqv{x(u_i,\cdot)}{x(u_j,\cdot)}_t\geqslant\\
&\geqslant 2t-2\int\limits_0^t \Gamma_0(x(u_i,s)-x(u_j,s))\,ds=\\
&=2\int\limits_0^t \left[1-\Gamma_0(x(u_i,s)-x(u_j,s))\right]\,ds\geqslant\\
&\geqslant 2(1-\sup_{\abs{z}\geqslant\delta_0} \Gamma_0(z)) \cdot t\rightarrow +\infty,\quad t\to +\infty,
\end{align*}
which contradicts the almost sure finiteness of $\overline{\tau}$.
\end{proof}

\begin{lemma}
\label{lemma10}
With probability one for any $i,j\in \{1,\ldots,n\}$ we have
$$
\lambda(\{t\geqslant 0 \mid x(u_i,t)-x(u_j,t)\in D\setminus \{0\}\})=0,
$$
where $\lambda$ is the one-dimensional Lebesgue measure and $D$ is defined in~\eqref{equation5}.
\end{lemma}

\begin{proof}
Let us fix $i,j\in\{1,\ldots,n\}$, $i>j$, and $z\in (0;u_i-u_j)$ and set
$$
\sigma_z:=\sup\{t\geqslant 0 \mid x(u_i,t)-x(u_j,t)\geqslant z\}.
$$
From Lemma~\ref{lemma9} it follows that $\sigma_z$ is finite almost surely. Also let $\tau_z$ be the restriction of the (random) mapping $\tau\colon [0;+\infty) \rightarrow [0;+\infty)$ defined in~\eqref{equation11} to the set $[0;\sigma_z]$ and $\tau_z^{-1}$ be its inverse. Then using~\eqref{equation10} we get
\begin{gather*}
\lambda(\{t\geqslant 0 \mid x(u_i,t)-x(u_j,t)\in D \cap [z;+\infty)\})=\\
=\lambda(\{0\leqslant t\leqslant\sigma_z \mid \beta(\tau(t))\in D_{ij}(z)\})=\lambda(\tau_z^{-1}(C_{ij}(z))),
\intertext{where}
D_{ij}(z):=D \cap [z-(u_i-u_j);+\infty)
\intertext{and}
C_{ij}(z):=\{t\geqslant 0 \mid \beta(t)\in D_{ij}(z)\}.
\end{gather*}

Moreover, since the stochastic process $\{x(u_i,t)-x(u_j,t),\; t\geqslant 0\}$ is a non-negative (continuous) martingale, we have
$$
r_z:=\inf\{x(u_i,t)-x(u_j,t) \mid 0\leqslant t\leqslant\sigma_z\}>0,
$$
and so by Lemma~\ref{lemma6}
$$
\rho_z:=1-\sup_{\abs{z'}\geqslant r_z} \Gamma_0(z')>0.
$$
Thus, we obtain that for any $s,t\in [0;\sigma_z]$, $s<t$, we have
$$
2\geqslant\dfrac{\tau(t)-\tau(s)}{t-s}\geqslant\dfrac{1}{t-s} \int\limits_s^t \left[1-\Gamma_0(x(u_i,r)-x(u_j,r))\right]\,dr\geqslant 2\rho_z.
$$
Therefore, for any $s,t\in [0;\tau(\sigma_z)]$, $s<t$,
$$
\dfrac{1}{2}\leqslant\dfrac{\tau_z^{-1}(t)-\tau_z^{-1}(s)}{t-s}\leqslant \dfrac{1}{2\rho_z}.
$$
This implies that the function $\tau_z$ is absolutely continuous and so it maps the sets of zero Lebesgue measure to the sets with the same property. Thus, from
$$
\lambda(C_{ij}(z))=0
$$
it follows that
$$
\lambda(\{t\geqslant 0 \mid x(u_i,t)-x(u_j,t)\in D \cap [z;+\infty)\})=0.
$$

Finally, since $z\in (0;u_i-u_j)$ was arbitrary and $x(u_i,\cdot)-x(u_j,\cdot)\geqslant 0$, we can conclude that
\begin{gather*}
\lambda(\{t\geqslant 0 \mid x(u_i,t)-x(u_j,t)\in D\setminus \{0\}\})=\\
=\lambda\left(\bigcup_{k\geqslant 1} \{t\geqslant 0 \mid x(u_i,t)-x(u_j,t)\in D \cap [1/k;+\infty)\}\right)=0.
\end{gather*}
The assertion of the lemma now follows trivially.
\end{proof}

To finish the proof of Theorem~\ref{theorem4} (obviously, it is enough to consider the case when $u_1<\ldots<u_n$) suppose that $(x(u_1,\cdot),\ldots,x(u_n,\cdot))$ is one of the weak limits (as $\ve\to 0+$) of the family $\{\vec{x}_\ve=(x_\ve(u_1,\cdot),\ldots, x_\ve(u_n,\cdot)),\; \ve>0\}$. Then for any $i,j\in\{1,\ldots,n\}$, $i>j$, the stochastic process $\{x(u_i,t)-x(u_j,t),\; t\geqslant 0\}$ is a non-negative martingale and so does not leave zero after hitting it. Since both $x(u_i,\cdot)$ and $x(u_j,\cdot)$ are standard Brownian motions, this implies that
$$
\jqv{x(u_i,\cdot)}{x(u_j,\cdot)}_t\geqslant\int\limits_0^t \1\{x(u_i,s)=x(u_j,s)\}\,ds,\quad t\geqslant 0.
$$
However, from Lemma~\ref{lemma8} (with $s=0$) and Lemma~\ref{lemma10} it follows that
$$
\jqv{x(u_i,\cdot)}{x(u_j,\cdot)}_t\leqslant \int\limits_0^t \1\{x(u_i,s)=x(u_j,s)\}\,ds,\quad t\geqslant 0.
$$
Hence
$$
\jqv{x(u_i,\cdot)}{x(u_j,\cdot)}_t=\int\limits_0^t \1\{x(u_i,s)=x(u_j,s)\}\,ds,\quad t\geqslant 0.
$$

Thus, we conclude that any weak limit (as $\ve\to 0+$) of the family $\{\vec{x}_\ve,\; \ve>0\}$ coincides in distribution with the $n$-point motion of the Arratia flow, which means that this family converges weakly to the latter.

\textbf{Acknowledgements.} The author is grateful to the anonymous referee for the careful reading of the paper and useful comments.

\end{document}